\documentclass[12pt,letterpaper]{amsart}
\usepackage{ amsmath,amsthm, amsbsy,amsfonts,amssymb, txfonts}
\usepackage{stmaryrd,mathrsfs,graphicx, amscd, tikz-cd, tikz, cancel}
\usepackage{tikz}
\usetikzlibrary{matrix,arrows,decorations.pathmorphing}

\usepackage[active]{srcltx}
\allowdisplaybreaks
\numberwithin{equation}{section}

\usepackage[
	hypertexnames=false,
	hyperindex,
	pagebackref,
	pdftex,
	breaklinks=true,
	bookmarks=false,
	colorlinks,
	linkcolor=blue,
	citecolor=red,
	urlcolor=red,
]{hyperref}
\usepackage{hyperref}

\def\BB{{\mathbb B}}

\def\CC{{\mathbb C}}

\def\FF{{\mathbb F}} 
 
\def\HH{{\mathbb H}}

\def\PP{{\mathbb P}}
\def\QQ{{\mathbb Q}} 
\def\RR{{\mathbb R}}

\def\ZZ{{\mathbb Z}}

\def\ssm{\smallsetminus}

\def\mmu{\boldsymbol{\mu}}

\def\G{\Gamma}
\def\g{\gamma}

\def\ram{{\rm rm}}

\def\sg{{\rm sg}}

\def\st{{\rm st}}

\def\bs{\backslash}
\newcommand{\eps}{\varepsilon}

\def\Dcal{{\mathcal D}}
\def\Ecal{{\mathcal E}} 
 
\def\Hcal{{\mathcal H}}

\def\Mcal{{\mathcal M}}

\def\Scal{{\mathfrak S}}

\def\hyp{{\mathcal H}}

\newcommand\aut{\operatorname{Aut}}

\newcommand\Diff{\operatorname{Diff}}

\newcommand\Epi{\operatorname{Epi}}

\newcommand\Hom{\operatorname{Hom}}

\newcommand\Mod{\operatorname{Mod}}

\newcommand\PSL{\operatorname{PSL}}
\newcommand\PGL{\operatorname{PGL}}

\newcommand\PU{\operatorname{PU}}
\newcommand\Sp{\operatorname{Sp}}

\newtheorem{theorem}{Theorem}[section]
\newtheorem{lemma}[theorem]{Lemma}
\newtheorem{proposition}[theorem]{Proposition}
\newtheorem{corollary}[theorem]{Corollary}

\numberwithin{definition}{section}

\theoremstyle{remark}

\newtheorem{remark}[theorem]{Remark}

\title[Moduli space of genus 4 curves]{A ball quotient parametrizing trigonal genus  4 curves}
\author{Eduard Looijenga}
\thanks{The author was supported by the Jump Trading Mathlab Research Fund.}
\address{University of Chicago (USA) and Universiteit Utrecht (Nederland)}

\begin{document}
\begin{abstract}
We consider the moduli space of genus 4 curves endowed with a $g^1_3$ (which maps with degree 2 onto the  moduli space of genus 4 curves). We prove that it defines a degree  $\frac{1}{2}(3^{10}-1)$ cover of the 9-dimensional Deligne-Mostow  ball quotient such that the natural divisors that live on that moduli space become totally geodesic (their normalizations are  8-dimensional ball quotients).

This isomorphism differs from the one considered by S.\ Kond\=o and its construction is perhaps more elementary, as it does not involve K3 surfaces and their Torelli theorem: the Deligne-Mostow ball quotient parametrizes  certain cyclic covers of degree 6 of a projective line and we show how a level structure on such a cover produces a degree 3 cover of that line with  the same discriminant, yielding a genus 4 curve endowed with a $g^1_3$.
\end{abstract}

\maketitle
\section*{Introduction}
Shigeyuki  Kond\=o constructed in \cite{kondo} an isomorphism from a Zariski open subset of the moduli space of 
nonsingular complex  genus four curves $\Mcal_4$ onto a Zariski open subset of a ball quotient, where the latter covers the  
9-dimensional ball quotient that appears in the Deligne-Mostow list. He also noted that this cannot be  extended to the full 
$\Mcal_4$, but that this is possible if we  blow up the hyperelliptic locus $\Hcal_4$ (which is of codimension 2 in $\Mcal_4$). The exceptional divisor then maps to a hyperball quotient, a property that also holds for the generic point of the thetanull locus. His  approach is based on the fact that the generic point of $\Mcal_4$ is represented by a curve $X$ of bidegree $(3,3)$ on $\PP^1\times \PP^1$. The $\mmu_3$-cover of $\PP^1\times \PP^1$ that totally ramifies along $X$ is a K3  surface with $\mmu_3$-action  and the Torelli theorem for such surfaces implies that its $\mmu_3$-Hodge structure  is a complete invariant of that $\mmu_3$-surface and hence also of $X$ (as the fixed point set of the $\mmu_3$-action). Such $\mmu_3$-Hodge structures are parametrized by a ball quotient of the above type.

The goal of this note is twofold. First to observe that there is a  modular  interpretation 
of the  hyperelliptic blowup  as the moduli space of  pairs  consisting of a nonsingular complex  
genus four curve  and a  $g^1_3$ on that curve, where we note that a generic genus 4 curve has \emph{two} $g^1_3$'s and that this moduli space comes with an involution that exchanges them.

  And  second, to give this moduli  space a complex hyperbolic structure  without resorting to 
K3 surfaces and their Torelli theorem. For this, we start from the fact that the 
the generic point of the  Deligne-Mostow ball quotient parametrizes $\mmu_6$-covers 
$C\to \PP^1$ totally ramified over a 12 element subset $D\subset \PP^1$ and nowhere else. 
We then show how a certain level structure on $C$ gives rise to 
a genus $4$ curve $X$ lying with degree 3 over $\PP^1$ with discriminant $D$ (so  $X$ comes with a $g^1_3$) and a correspondence between $C$ and $X$. 

Our approach is in the same spirit as our paper with Heckman \cite{HL} on the moduli space of rational elliptic surfaces (which here appears  in the guise of the thetanull locus).  
Most of that paper  was concerned with the construction and comparison of certain 
compactifications  of that moduli space  (which accounts for its length) and it is likely that a 
similar program could be pursued here, where,  of course, a compactification 
which parametrizes curves of arithmetic genus 4 endowed with a $g^1_3$ would be best.

Let us finally point out that our period map is quite different from the one introduced  by Kond\=o, as the two discriminants that we associate  to the two $g^1_3$'s on a generic genus $4$ 
curve will not lie in the same $\PGL_2(\CC)$-orbit and hence have different image in the  Deligne-Mostow ball quotient. Presumably the two period maps do not even coincide on the generic point of the thetanull  locus, where the genus 4 curve has only one $g^1_3$.

\section{The nine-dimensional Deligne-Mostow ball quotient}\label{sect:DMball} 
This section is mostly a review of known results. Its main purpose is to recall them in a way that suits our purpose and to establish notation.

\subsection*{$\mmu_6$-covers of $\PP^1$}
Let $C$ be a nonsingular complex-projective curve endowed with an $\mmu_6$-action that has 12 fixed points, is free 
elsewhere and is such that its orbit space $P$ is a copy of $\PP^1$. The  discriminant of $C\to P$ is then of the form $5D_C$, 
with $D_C$ a $12$-element subset of $P$, considered as a reduced divisor and Riemann-Hurwitz shows that $C$ will have genus $25$. 
Note that the pair $(P,D_C)$ determines $C$ uniquely, the isomorphism  being unique up to a covering transformation 
(in $\mmu_6$). We recall how this gives rise to a point in a $9$-dimensional ball quotient, which happens to be the one of highest dimension  that  appears in the Deligne-Mostow list. 
 
Since we find it convenient to make a distinction between what depends on the complex 
 structure and what only depends on the underlying topology, we also fix  a closed oriented 
 surface $\Sigma$ for  of genus 25 endowed with a $\mmu_6$-action of the type above, so  with $12$ fixed points and acting freely elsewhere.  We denote by $\pi:\Sigma\to S$ for the formation of its $\mmu_6$-orbit  space (a smooth $2$-sphere), and by $D\subset S$ its set of critical values. Then any 
 complex structure on $S$ (which will make it a copy of the Riemann sphere) determines a unique one on $\Sigma$ for which the $\mmu_6$-action and $\pi$ are holomorphic, yielding a $\mmu_6$-curve as above with discriminant $5D$. We also note that a diffeomorphism of $S$ which preserves $D$ lifts to $S$ and does so almost uniquely: two such lifts will lie in the same $\mmu_6$-orbit.

\subsection*{Hodge structure of a $\mmu_6$-cover  of $\PP^1$}
We first note that the standard embedding $\chi: \mmu_6\subset \CC$ and its complex conjugate are the only two primitive characters of $\mmu_6$. The group $\mmu_6$ acts on  $H^1(\Sigma; \CC)$ and decomposes the latter into its character spaces: $H^1(\Sigma; \CC)=\oplus_{i\in \ZZ/6} H^1(\Sigma; \CC)_{\chi^i}$.  The intersection pairing gives rise to a nondegenerate Hermitian form on $H^1(\Sigma; \CC)$ defined by 
$(\alpha, \beta)\mapsto \sqrt{-1}\int_\Sigma \alpha\wedge\overline\beta$.  
It has signature $(25,25)$ and for this form the above decomposition into character spaces is perpendicular. In the presence of a complex structure it is also compatible with the Hodge structure.
For example, the decomposition  
\begin{equation}\label{Hdec}
H^1(C; \CC)_{\chi}=H^1(C, \Omega_C)_{\chi}\oplus \overline{H^1(C, \Omega_C)_{\overline\chi}}.
\end{equation}
is perpendicular with respect to this hermitian form with the first summand being positive definite of dimension one and the second  summand negative definite of dimension nine. Since the group of covering transformations $\mmu_6$ acts trivially on $\PP (H^1(\Sigma; \CC)_\chi)$ and so this projective space  only depends on the pair $(S,D)$; we therefore denote  it by $\PP(S,D)$.
Its group $\PU (H^1(\Sigma; \CC)_\chi)$ of projective unitary transformations, which we denote for a similar reason by $\PU(S,D)$, is isomorphic to $\PU (1,9)$. The positive definite complex lines in $H^1(\Sigma; \CC)_\chi$ make up an open complex ball $\BB(S,D)\subset \PP(S,D)$ on  which $\PU (S,D)$ acts transitively with compact stabilizers; indeed, this ball  is the symmetric domain associated to $\PU (S,D)$. If there is no risk of confusion, we simply write $\BB$ for $\BB(S,D)$.

As we observed above, a complex structure on $S$ determines one on $\Sigma$, turning it into a compact Riemann surface $C$ with $\mmu_6$-action. The Hodge decomposition (\ref{Hdec}) then defines an element, the \emph{Hodge point}, of $\BB$. 
The main result of Deligne and Mostow \cite{DM}, \cite{mostow} implies that this  Hodge point is a complete invariant of the relative isotopy class of this complex structure.  

\subsection*{Action of a mapping class group}
To make this last  assertion precise, let $\Diff^+(S,D)$ denote the group of orientation preserving diffeomorphisms of $S$ which preserve $D$.  Since  this group $\Diff^+(S,D)$ acts on $\Sigma$ up to covering transformations, it gives rise to a well-defined action on $\PP(S,D)$ via  $\PU (S,D)$ (and so
preserves $\BB$). This action factors through its connected component group that we shall denote $\Mod(S,D)$ and which we readily recognize as the spherical braid group on 12 strands. That group has as a distinguished conjugacy class the set of  simple braids.  We recall that a \emph{simple braid} is given by  an (unoriented) arc $\delta$ in 
$S$ which connects two points of $D$ and has no points of $D$ in its interior: the associated simple braid 
$T_\delta\in \Mod(S,D)$ has its support in a regular neighborhood of $\delta$ and lets $\delta$ make a half-turn in a counterclockwise direction. 
These simple braids generate $\Mod(S,D)$. In fact, if we enumerate the points of $D$ by $\{p_i\}_{i\in\ZZ/12}$ and let $\delta_i$ connect $p_i$ with $p_{i+1}$ in such a way as that their juxtaposition produces an embedded circle in $S$, then
$T_{\delta_1},\dots, T_{\delta_{10}}$ already generate $\Mod(S,D)$.
It follows from the work of Deligne-Mostow that the image of $\Mod(S,D)$  in $\PU (S,D)$ is an arithmetic group, which therefore acts properly discretely on $\BB$ with finite covolume. 
We denote this image group by  $\G\subset \PU (S,D)$. 

\subsection*{A lattice over the Eisenstein ring}
In order to make the action of a simple braid on $\BB$ somewhat more explicit, we make some general observations first. Write $\tau$ for the primitive $6$th root of unity with positive imaginary part, $e^{\pi\sqrt{-1}/3}$,  and regard it as a generator of $\mmu_6$. As a complex number it satisfies  $\tau^2=\tau-1$, but this identity is of course not valid in the group ring  $\ZZ[\mmu_6]$. Given a $\ZZ[\mmu_6]$-module $V$, let us denote by 
$V^\circ$ resp.\ $V_\circ$ the maximal submodule resp.\ quotient module on which $\tau$ satisfies the identity $\tau^2=\tau-1$.
For example, $V_\circ$ is the quotient of $V$ by the subgroup of $v\in V$ with nontrivial 
$\mmu_6$-stabilizer. Then  $V^\circ$ and $V_\circ$ become modules over the ring of \emph{Eisenstein integers} $\Ecal:=\ZZ+\ZZ\tau \subset \CC$, the ring of integers the cyclotomic field $\QQ(\sqrt[3]{1})$. Note that $\RR\otimes V^\circ\to\RR\otimes V_\circ$ is an isomorphism and can be identified with the summand of  $\RR\otimes V$ which decomposes into planes on which $\tau$ acts as a rotation over $2\pi/6$.
Since $\Ecal$ is a principal ideal domain, any finitely generated torsion free $\Ecal$-module is in fact free. So if $V$ is a finitely generated free $\ZZ$-module, then both $V_\circ$ and $V^\circ$ are  finitely generated free $\Ecal$-modules. If  we are in addition given a $\mmu_6$-invariant antisymmetric pairing $a: V\times V\to \ZZ$, then this gives rise to
the $\frac{1}{2}\Ecal$-valued hermitian form on $V^\circ$ defined by 
\[
h_a(x,y):=\tfrac{1}{2}\tau a(x, y)-\tfrac{1}{2}a(x,\tau y)
\]
(the justification for the perhaps somewhat unexpected normalization factor $\tfrac{1}{2}$ will be given below). Note that $h_a(x,x)=-\tfrac{1}{2}a(x,\tau x)$ and 
$h_a(x,y)-\overline{h_a(x,y)}=\tfrac{1}{2}(\tau-\overline\tau) a(x, y)=\tfrac{1}{2}\theta a(x,y)$, where 
\[
\theta:=\tau-\overline\tau=-1+2\tau=\sqrt{-3}
\]
is the `first' purely imaginary number in $\Ecal$.
We apply this to the $\ZZ[\mmu_6]$-module $H_1(\Sigma)$.  We find that $H_1(\Sigma)^\circ$  and $H_1(\Sigma)_\circ$ are free  $\Ecal$-modules of rank 10. The intersection pairing on $H_1(\Sigma)$ is nondegenerate on $H_1(\Sigma)^\circ$
and defines via the receipe above on this lattice the $\Ecal$-valued hermitian form 
\[
h(x,y):=\tfrac{1}{2}\tau (x\cdot y)-\tfrac{1}{2}(x\cdot\tau y). 
\]
We will write $L$ for $H_1(\Sigma)^\circ$ endowed with this form.

Let us see what this induces on the preimage  $\pi^{-1}\delta$ of an arc $\delta$ as above.
We now assume $\delta$ oriented and  let $\tilde\delta$ be a lift of $\delta$ over $\pi$. Then $c_\delta:=(1-\tau^3)\tilde\delta$ is an embedded oriented circle (hence a $1$-cycle) and meets its $ \tau$-translate in 2 points, each with multiplicity $1$. In other words,
$c_\delta\cdot \tau c_\delta=2$.
Now $H_1(\pi^{-1}\delta)^\circ$ has the $\Ecal$-generator $a_\delta:=(1+\tau)(1-\tau^3)\tilde\delta=(1+\tau)c_\delta$ (which is unique up multiplication by an element of $\mmu_6$), and the above formula shows that $h(a_\delta,a_\delta)=\tfrac{1}{2}.-6=-3$. A similar computation shows that $h(a_{\delta_i},a_{\delta_{i+1}})\in\mmu_6\theta$. So $h|L\times L$ takes values in $\theta\Ecal$, 
so that  $\theta^{-1}h$ defines a skew-hermitian form 
\[
h':=\theta^{-1}h: L\times L\to \Ecal
\]
Since we have the freedom of multiplying  each $a_\delta$ with an element of $\mmu_6$, we can choose $a_{\delta_1},\dots , a_{\delta_{10}}$  in such a manner that it is a $\Ecal$-basis of $L$ with 
\[
h(a_{\delta_i}, a_{\delta_j})=
\begin{cases} 
-3 &\text{ when $j=i$,}\\
\pm \theta &\text{ when $j=i\pm 1$,}\\
0 &\text{ otherwise.}\\
\end{cases}
\]
This also shows  that $L$ is isomorphic to the lattice thus denoted in \cite{allcock}  (this explains the coefficient $\tfrac{1}{2}$ in our definition of $h$) and if we  $h$ replace by $-h$, it becomes isomorphic to the lattice denoted  $\Lambda$ in \cite{HL}.  In \cite{allcock},   $L$  is in fact constructed from the
even $\ZZ$-lattice $\mathbf{E}_8\perp\mathbf{E}_8\perp\mathbf{U}\perp\mathbf{U}$. Going in the opposite direction, this amounts to the statement  that we recover this lattice if take the underlying abelian group and the even symmetric pairing on it defined by the real part of $h$ multiplied with $-\frac{2}{3}$. 

\subsection*{Mirror hyperballs}
The simple braid $T_\delta$ acts in  $L$ as  the unitary transformation $s_\delta$  defined by   
\[
s_\delta(x)=x+\tfrac{1}{3}(1-\tau^2)h(x,a_\delta )a_\delta=  x+\tau h'(x,a_\delta)a_\delta
\]
It is what is called in \cite{allcock}  a \emph{triflection} in $\Ecal a_\delta$: it multiplies  $a_\delta$ with the third root of unity $\tau^2$ and is the identity on the orthogonal complement of $a_\delta$. It indeed preserves $L$, as  follows from the fact  that $h'$ takes its values in $\Ecal$.

Ordinary duality identifies $H^1(\Sigma)^\circ$ resp.\ $H^1(\Sigma)_\circ$ with the dual of 
$H^1(\Sigma)_\circ$ resp.\ $H^1(\Sigma)^\circ$. We use Poincar\'e  duality to identify these with 
resp.\  $L$ resp.\ $H_1(\Sigma)_\circ$ and use on these the same hermitian form as above.
 This identifies  $\PP(S,D)$ with $\PP(\CC\otimes_\Ecal L)$. In particular, $T_\delta$ acts on 
 $\BB$ as a triflection.  So its fixed point set $\BB^{T_\delta}$ is a hyperball (to which we shall 
 also refer as a \emph{mirror} in $\BB$). Since the $T_\delta$ make up a single conjugacy class in 
$\Mod(S,D)$, the group $\Mod(S,D)$ permutes the mirror hyperballs transitively.

The group of unitary transformations of  the $\Ecal$-module $L$  acts properly discretely on $\BB$ with kernel its center $\mmu_6$. The group $\G$ is a priori contained in the quotient of this group by its center (which is an arithmetic subgroup of $\PU(S,D)$), but as Allcock  has shown (\cite{allcock}, Thm.\ 5.1), is in fact equal to it. So $\G$ acts properly discontinuously on  $\BB$ and the Baily-Borel theory asserts that $\BB_\G:=\G\bs\BB$  has the structure of a quasi-projective orbifold that can be completed to a normal projective variety by adding a finite set of \emph{cusps} (in this case, just one).
The image of the Hodge point in this ball quotient $\BB_\G$ is a complete invariant of the isomorphism type of the $\mmu_6$-curve $C$, or equivalently, of the pair $(S,D)$ with the given complex structure.

The $\Mod(S,D)$-centralizer of  $T_\delta$ acts on the mirror $\BB^{T_\delta}$ with image an arithmetic subgroup and 
the resulting ball quotient  of $\BB^{T_\delta}$ appears in the Deligne-Mostow list as the one defined by the sequence $(\frac{2}{12},\frac{1}{12}, \frac{1}{12},\dots ,\frac{1}{12})$. 

The group $\G$ is as a quotient of  $\Mod(S, D)$ obtained by imposing the relation $T_\delta^3=1$. The 
mirrors  are locally finite on $\BB$ and since they are transitively permuted by $\G$, they define 
an irreducible totally geodesic hypersurface $\Dcal$ in $\BB_\G$ that we shall call  the \emph{confluence locus}. 
We write $\BB^\circ$ for the complement of the union of the mirrors (an open subset of $\BB$). 
So this is the preimage of  $\BB_\G^\circ\ssm \Dcal$. 

\subsection*{The Deligne-Mostow  isomorphism}
Let $\Mcal(12)$ denote the $\PGL(2, \CC)$-orbit space of the configuration space of $12$-element subsets of $\PP^1$ (this is also equal to the $\Scal_{12}$-orbit space of $\Mcal_{0,12}$). The above construction defines a homomorphic map 
\[
\Mcal(12)\to \BB_\G.
\]
 If we think of a $12$-element subset of $\PP^1$ as a reduced degree 
$12$-divisor, then the above map extends to  the locus $\Mcal(12)^\st$ of Hilbert-Mumford stable positive degree $12$-divisors, 
that is, divisors for which the multiplicities are at most $5$. The Deligne-Mostow theory asserts that this  extension is in fact an isomorphism 
\[
\Mcal(12)^\st\xrightarrow{\cong}  \BB_\G
\]
that takes $\Mcal(12)$ onto  $\BB_\G^\circ$. It even extends to an isomorphism of their natural compactifications: the GIT compactification on the left and the Baily-Borel compactification on the right. Both are one-point compactifications, with the added point on the left being represented by a $6$-divisible divisor on $\PP^1$ whose support consists of two points.

\section{Trigonal structures on a genus four curve}
\subsection*{$\mathbf{g^1_3}$'s on a genus 4 curve}
Let $X$ be a smooth connected complex-projective  curve of genus $4$ and let $P$ be a pencil of degree $3$ on $X$ (in other words, a $g^1_3$).
If $P$ has no base point, then the pencil  defines a morphism $X\to P$ of degree $3$ and  by Riemann-Hurwitz, the discriminant divisor $D(X, P)$ of this morphism is of degree $12$. Note that its multiplicities will be $\le 2$.
In case $P$ has a base point, then the residual base point free pencil $P'$ will have degree $2$ and this makes $X$ hyperelliptic with 
$P'$ being the hyperelliptic pencil. We will see  that it is then natural to define the discriminant divisor of $P$ to be 
the degree $12$-divisor $D(X, P)$  that is (via the identification of $P$ with $P'$)  the sum of the discriminant divisor of $X\to P'$ (the image of the  Weierstra\ss\ points, so of degree 10) plus twice the image in $P'$  of the fixed point $P$. So its multiplicities are then $\le 3$.

Recall that if $X$ is nonhyperelliptic, then a canonical image of $X$ in $\PP^3$ is the transversal intersection of a cubic surface with a quadric $Q$ surface, the  latter being unique and either a smooth (so isomorphic to $\PP^1\times\PP^1$) or a quadric cone. In the first case, the projections on the two factors give us two pencils of degree $3$. In the  second case, projection away from the vertex of the cone also yields a  pencil of degree $3$; the line bundle is question then an effective  (even) theta characteristic. It is well-known that this exhausts all the $g^1_3$'s in genus 4. For future reference we state this as a lemma (and only sketch the proof).

\begin{lemma}\label{lemma:pencils}
Let $X$ be a smooth connected complex-projective  curve of genus $4$. If $X$ is nonhyperelliptic, 
then any $g^1_3$ on $X$ is as above: it comes with a residual $g^1_3$ (in the sense that its members supplement the members of the given $g^1_3$ to a canonical divisor) and there are no others (so the two being equal 
if and only if the $g^1_3$ is a theta characteristic). If $X$ is  hyperelliptic, then any $g^1_3$ on $X$ is the sum of an $x\in X$ and its hyperelliptic $g^1_2$; in that case the pencil is a theta characteristic if and only if $x$ is a Weierstra\ss\ point.
\end{lemma}
\begin{proof}
This is well-known, as it is in a sense a geometric formulation of Clifford's theorem in genus $4$. The proof runs as follows. If $P$ is a $g^1_3$ on $X$, then Clifford's theorem implies that $P$ is a complete linear system.
By Riemann-Roch we have a supplementary $g^1_3$  on $X$, denoted $P'$,  whose members supplement the 
members of $P$ to a canonical divisor. So each member of $P$ resp.\ $P'$ defines a line in 
$\check\PP(H^0(X, \Omega_X))\cong \PP^3$ which meets the canonical image of $X$ in at most three points and 
$P'$ resp.\ $P$ is realized by the pencil of planes through that line. We have $P=P'$ if and only if $P$ is a 
theta characteristic. The rest of the argument is left to the reader.
\end{proof}

\subsection*{A moduli stack  of genus 4 curves with a $\mathbf{g^1_3}$}
We interpret this in terms of moduli spaces. Pairs $(X, P)$ as above define a Deligne-Mumford stack that 
we shall denote by $\Mcal_4(g^1_3)$. It comes with an involution which assigns to $(X,P)$ the residual pair $(X,P')$ that is characterized by the property that $P+P'$ lies in the canonical system.
The forgetful morphism $\Mcal_4(g^1_3)\to \Mcal_4$ evidently  factors through the orbit space of this involution. 
Recall that the moduli space $\Mcal_4$ contains  the thetanull locus $\Mcal_4^0$ (the locus for which the curve admits an even effective 
theta characteristic) as a substack of codimension one and that this  substack $\Mcal_4^0$ in turn contains the hyperelliptic locus $\Hcal_4$  as a substack of codimension one. 

Consider the  Stein factorization 
\[
\Mcal_4(g^1_3)\to\hat\Mcal_4\to \Mcal_4
\]
It follows from Lemma \ref{lemma:pencils} that $\hat\Mcal_4\to \Mcal_4$ is finite of degree two and ramified along $\Mcal_4^0$, with the covering transformation being induced by the residual involution. It also shows that if $\hat\Hcal_4$ stands for the (isomorphic) preimage of  $\Hcal_4$ in $\hat\Mcal_4$, then $\Mcal_4(g^1_3)\to\hat\Mcal_4$ is an isomorphism over $\hat\Mcal_4\ssm \hat\Hcal_4$ and is over 
$\hat\Hcal_4\cong\Hcal_4$ the universal hyperelliptic curve $\Hcal_4(g^1_3)$ of genus $4$ (so that will be a divisor). The strict transform of  $\Mcal_4^0$ in $\Mcal_4(g^1_3)$ is another  divisor $\Mcal_4(g^1_3)^0$ on $\Mcal_4(g^1_3)$ which meets 
$\Hcal_4(g^1_3)$ in  the  locus  parametrizing the Weierstrass points  (this locus is geometrically connected, since the monodromy is transitive on the Weierstrass points).

Our goal is to realize  $\Mcal_4(g^1_3)$ as an open subset of a  quotient of $\BB$ relative some  finite index subgroup of $\G$. We determine this subgroup in two ways: first in a purely topological manner and subsequently more in the spirit of algebraic geometry as a level structure on the $\mmu_6$-curve $\Sigma$.

\subsection*{The approach via a classification of coverings} This begins with  addressing  the question of how many topological types of smooth connected degree three coverings of the $2$-sphere $S$ exist that have the $12$-element subset $D\subset S$ as reduced discriminant. Such a covering is given by its restriction to $S\ssm D$ (which is then unramified).  So if we choose a base point $o\in S\ssm D$, then it  can be given by a group homomorphism $\rho: \pi_1(S\ssm D, o)\to \Scal_3$ which assigns to a simple loop around a point of $D$ a transposition. The image of $\rho$ must be a transitive subgroup of $\Scal_3$ and hence must be all of $\Scal_3$. Two such epimorphisms define isomorphic coverings if and only if they differ by an inner 
automorphism of $\Scal_3$. So if we let  $\Epi'(\pi_1(S\ssm D, o),\Scal_3)$ stand for the set of group of surjective homomorphisms $\pi_1(S\ssm D, o)\to \Scal_3$ that take simple loops to transpositions, then the  set of topological types is naturally identified with
\[
R(S,D):= \Scal_3\bs \Epi'(\pi_1(S\ssm D, o),\Scal_3).
\]
We determine its number of elements. We choose smooth arcs $\{\g_i\}_{i\in\ZZ/12}$ from $o$ to the distinct points of $D$ that only meet at $o$ and depart from there along  rays  in $T_oS$ in a counterclock wise order. This means that the associated simple loop associated to $\g_i$  defines a $c_i\in\pi_1(S\ssm D, o)$ so that $c_{11}\cdots c_0=1$ and  $c_1, \dots , c_{11}$ are free generators. Then a $\rho$ as above is given by its values on $c_1, \dots , c_{11}$.
Note that $\rho(c_1), \dots , \rho(c_{11})$  is a sequence in the 3-element set of transpositions $\{(12), (23), (31)\}$ of $\Scal_3$ 
that cannot be nonconstant, but can otherwise be arbitrary. So there are $3^{11}-3$ such $\rho$. The action of $\Scal_3$  on this set by conjugation is free and so  $\#R(S,D)=(3^{11}-3)/6=(3^{10}-1)/2=3^9+3^8+\cdots +1$. This number 
is also equal to $\#\PP^9(\FF_3)$. We shall see that this is not a  coincidence.

Evidently the group of automorphisms of  $\pi_1(S\ssm D, o)$ which preserves the simple loops 
acts on $R(S,D)$. Since we have divided out by   $\Scal_3$-conjugation, the inner automorphisms 
act trivially, so that we have an outer action of  $\Mod(S,D)$  on $\pi_1(S\ssm D, o)$. 
This gives us  a genuine action of  $\Mod(S,D)$ on $R(S,D)$. 

\begin{lemma}\label{lemma:transitive}
The action of $\Mod(S,D)$ on $R(S,D)$ is  through $\G$ and is transitive. 
\end{lemma}
\begin{proof}
Recall that an (unoriented) arc $\delta$ in $S$ connecting two distinct points of $D$ whose 
relative interior does not meet any point of $D$  represents a simple braid in $\Mod(S,D)$. If we choose an interior  point of $\delta$ as our base point $o$, then this yields two simple  loops and hence two transpositions of the fiber over $o$.
Then $T_\delta$  acts trivially on $\rho$ if these two  transpositions are equal and, as a straightforward verification shows,  is of order $3$ otherwise. It follows that this action of $\Mod(S,D)$ on $R(S,D)$ factors through $\G$. 

For the transitivity property it is enough to show that for a  given degree 3 covering $\Sigma'\to S$ as above there exists a standard generating set  $(c_0, \dots , c_{11})$ of $\pi_1(S\ssm D, o)$ as above and a numbering 
of the fiber over $o$ such that $\rho_{\Sigma'}(c_i)$ equals $(12)$ for $i$ even and  $(23)$ for $i$ odd. 
Since $\Mod(S,D)$ acts transitively on the collection of standard generating sets, this will indeed suffice.

Suppose that for $i\le 10$ we have constructed arcs $\g_1, \dots , \g_{i}$ that depart at $o$  along tangent rays in a counterclockwise order, are otherwise disjoint and are such that $\rho(c_1), \dots \rho(c_i)$ are as desired, and with the property that
the preimage in $\Sigma'$ over $\Delta_i:=S\ssm \{\g_1\cup \dots \cup \g_i\}$ (a copy of an open disk)  is connected. 
This means that the monodromy over $\Delta_i$ is still the full $\Scal_3$. The reader will then 
have no trouble finding an arc $\g_{i+1}$ which departs along a ray in the sector spanned by the rays of departure of  $\g_i$ and $\g_1$ and which stays in $\{0\}\cup\Delta$ and whose associated monodromy is the prescribed transposition. 

But if $i\le 9$,  extra care is needed to  ensure that the preimage in $\Sigma'$ over 
$\Delta_{i+1}:=S\ssm \{\g_1\cup \dots \cup \g_{i+1}\}$ is still connected. That this indeed possible follows from the fact that the preimage over $\Delta_i$  is connected  and has at least three  points of simple ramification.

On the other hand, if $i=10$  we are done, for then the value of $\rho$ on  $c_0:=c_1^{-1}\cdots c_{11}^{-1}$ will be 
$\big((23)(12)\big)^5(23)=(12)$. This completes the proof.
\end{proof}

Fix a degree 3 covering $\Sigma'\to S$ as above and denote by  $r\in R(S,D)$ the associated element and denote by $\G'\subset \G$ its stabilizer. It follows from Lemma \ref{lemma:transitive} that $\G'$ is a subgroup $\G$ of index 
$(3^9-1)/2$. By construction,  the isomorphism type of a connected smooth degree $3$ covering of a 
copy of $\PP^1$ with reduced discriminant of degree $12$ determines a $\G'$-orbit in $\BB^\circ$. 
So if $\Mcal_4(g^1_3)^\circ$ stands for the  open subset of $\Mcal_4(g^1_3)$ that parametrizes pairs $(X, P)$  for which $P$ has precisely $12$ nonreduced members (in other words, for which $X$ is nonhyperelliptic and for which $X\to P$ has reduced discrimininant), then we find:

\begin{corollary}\label{cor:ballquotient1}
The open subset $\Mcal_4(g^1_3)^\circ$ is naturally isomorphic to the ball quotient $\BB_{\G'}^\circ$.
\end{corollary}

In order to get all of $\Mcal_4(g^1_3)$ we must allow the discriminant divisor $D$ to become nonreduced. What we must exclude is that 
two points of $D$ lie in a disk $\Delta\subset P$ (which contains no other points of $D$) over whose  boundary the covering $X\to P$ is trivial, for if we let these points then coalesce inside $\Delta$ to a point $z\in \Delta$, this 
will create an ordinary double point of the covering. 
There is however one case that we should not throw out: if the covering has in addition a section over $P\ssm \Delta$, 
then the degeneration will be the union of a smooth double covering $X'\to P$ (whose discriminant is 
$D\cap (P\ssm \Delta)$) and a component $\tilde P$ which maps isomorphically onto $P$ and meets $X'$ transversally over  $z$. So this amounts to specifying a hyperelliptic curve of genus $4$ and a point on that curve; in other words, and element of $\Hcal_4(g^1_3)$.
This happens for example when  $\rho(c_0)=\rho(c_1)\not=\rho(c_3)=\dots =\rho(c_{11})$ and 
we take for $\Delta$ a thin regular neighborhood $\g_0\cup \g_1$.  

We can state this in terms of the $\G$-action on $R(S,D)$. The confluence of two points along an arc $\delta$ as above represents  a triflection in $\G$. It may or may not act trivially on $R(S,D)$. If the action is nontrivial, then 
this confluence  creates a point of total ramification and we still have defined a genus $4$ covering of $S$. Otherwise
this creates an ordinary double point. These three cases translate into saying that the stabilizer $\G'$
has three orbits in the collection of mirrors. Equivalently, the preimage of the confluence divisor $\Dcal$ in $\BB_\G$ under the natural map 
\[
\BB_{\G'}\to \BB_\G
\]
has three irreducible components $\Dcal_{\ram}, \Dcal_{\sg},\Dcal_{\hyp}$ whose generic points are characterized topologically by the degree 3 covering  acquiring  respectively (\emph{rm}) a point of total ramification, (\emph{sg}) an ordinary  double point whose complement is connected and ($\hyp$)  an ordinary  double point which disconnects. We can can now improve upon Corollary \ref{cor:ballquotient1} as follows.

\begin{theorem}\label{thm:ballquotient2}
The  above construction gives an identification 
\[
\Mcal_4(g^1_3)\cong \BB_{\G'}\ssm \Dcal_{\sg},
\]
thus endowing  the left hand side with an (incomplete) complex-hyperbolic metric.
This identifies the  preimage of $\Dcal_{\hyp}$ with the universal hyperelliptic curve of genus $4$. 
\end{theorem}

For a different approach to this theorem (that  ultimately yields a more precise statement), we first need to discuss level 3 structures for the Eisenstein lattice $L$.

\subsection*{The approach via a  level structure}
We  begin with the observation that $\Ecal/\theta\Ecal$ is a $\FF_3$-vector space of dimension one on which $\tau\in\mmu_6$ acts as minus the identity (because $\tau +1=\tau^{-1}\theta$).
It follows that $\FF_3\otimes L=L/\theta L$ is a $\FF_3$-vector space of rank 10 on which the skew-hermitian form $h'$ induces a symplectic form
$h'_{\FF_3}:=\FF_3\otimes h'$. The following lemma generalizes an argument of Allcock (appearing in the proof of Thm.\ 5.2 of \cite{allcock}).

\begin{lemma}\label{lemma:mod3}
The symplectic form $h'_{\FF_3}$ is nondegenerate on  $\FF_3\otimes L$ (so that it is isomorphic to $\FF_3^{10}$ equipped with its standard symplectic form)  and $\G$ acts on $\FF_3\otimes L$  with image its full symplectic group (a copy of 
$\Sp_{10}(\FF_3)$). The image  of a triflection is a symplectic transvection.
\end{lemma}
\begin{proof}
The basis $a_{\delta_1},\dots , a_{\delta_{10}}$ of $L$ maps to a basis $\alpha_1, \dots, \alpha_{10}$ of $\FF_3\otimes L$ with the property that 
$h'_{\FF_3}(\alpha_i, \alpha_j)$ is $\pm 1$ if $j=\pm i$ and is zero otherwise. The  triflection defined by $a_{\delta_i}$ becomes $x\mapsto x+\tau h'_{\FF_3}(x, \alpha_i)\alpha_i=  x-h'_{\FF_3}(x, \alpha_i)\alpha_i$, which is indeed a  symplectic transvection. 
Let $V$ be  the free abelian group with basis $\tilde\alpha_1, \dots, \tilde\alpha_{10}$ endowed with the symplectic form that assigns to the pair  $(\tilde\alpha_i, \tilde\alpha_j)$ the value  $\pm 1$ if $j=\pm i$ and is zero otherwise, so that its reduction mod $3$ gives 
$h'_{\FF_3}$. This lattice is well-known in singularity theory (it is the one  for the 
Milnor fiber of a plane curve singularity of type $A_{10}$, defined by $z_1^2+z_2^{11}=0$ as given on 
a standard basis  of vanishing cycles). It was proved by A'Campo \cite {ac} that the subgroup of $\Sp(V)$ generated by these symplectic transvections contains the principal congruence level 2 subgroup. This implies that it maps onto the full symplectic  group of its mod $3$ reduction.
\end{proof}

The covering group $\mmu_6$ acts on 
the 10-dimensional $\FF_3$-vector space  $\FF_3\otimes L$ by scalars  and so the  
9-dimensional projective space $\PP(\FF_3\otimes L)$ only depends on the pair  $(S,D)$. 

We regard $L$ as a primitive subgroup of $H^1(\Sigma)$, so that a one dimensional  subspace  $\ell\subset \FF_3\otimes L$ defines (by restriction) a surjection 
$H_1(\Sigma)\cong \Hom (H^1(\Sigma), \ZZ) \to \Hom(\ell, \FF_3)=\ell^\vee$. This surjection is $\mmu_6$-equivariant if we let $\mmu_6$ act on $\ell^\vee$ via the obvious character $\mmu_6\twoheadrightarrow\mmu_2=\{\pm1\}$. This also yields an unramified $\ell^\vee$-covering 
$\Sigma_\ell\to \Sigma$.

\begin{lemma}\label{lemma:}
The $\mmu_6$-action on $\Sigma$ lifts to $\Sigma_\ell$. If $\Sigma'_\ell$ denotes the orbit
space of this lift, then $\Sigma'_\ell$ is an oriented surface of genus 4 such that the obvious map  
$\Sigma'_\ell\to S$ is orientation preserving of degree $3$ with discriminant $D$ (in other words, defines an element of $R(S,D)$). Moreover, this degree $3$ cover is  
is independent of the lift of the $\mmu_6$-action, in the sense that for two any lifts, the two degree 3 covers  of $S$ are isomorphic by a unique isomorphism.
\end{lemma}
\begin{proof}
We first show that the $\mmu_6$-action on $\Sigma$ lifts to $\Sigma_\ell$.
Recall that if we choose  a base point $x\in \Sigma$, then $\Sigma_\ell$ can be obtained as a 
quotient of the space of paths in $\Sigma$ that begin at $x$: two  such paths $\alpha, \beta$ 
define the same point of $\Sigma_\ell$ if and only if $\alpha$ and $\beta$ have the same endpoint 
and the class of the loop  $\beta^{-1}*\alpha$ in $\pi_1(\Sigma, x)$ has zero image in $\ell^\vee$. 
If we choose $x$ to be fixed point of the $\mmu_6$-action, then $\mmu_6$ will act on the 
space of such of paths. The homomorphism $\pi_1(\Sigma, x)\to \ell^\vee$ is $\mmu_6$-equivariant, in particular, the $\mmu_6$-action on $\pi_1(\Sigma, x)$ preserves its kernel. So $\mmu_6$ preserves the above equivalence relation, meaning that  we have lifted the  
$\mmu_6$-action on $\Sigma$ to $\Sigma_\ell$. 

This lift is not canonical, but two such lifts differ by an $\ell^\vee$-covering transformation.  
Hence the isomorphism type of the $\mmu_6$-orbit space $\Sigma_\ell'$  of $\Sigma_\ell$, now viewed as a  degree $3$ cover of $S$, is independent of the lift. Note that $\Sigma_\ell'$  is a connected oriented surface, not ramified  over 
$S\ssm D$. To see  what happens over $z\in D$, we observe that the  fiber $\Sigma_\ell(z)$ 
of $\Sigma_\ell\to S$ over $z$ has the structure of an affine line over $\ell^\vee$ on which 
$\mmu_6$ also acts  and in such a manner that the induced action on $\ell^\vee$ is minus the identity.
This means that $\mmu_6$  acts as transposition on $\Sigma_\ell(z)$: it fixes a point and 
exchanges the  remaining two. It follows  that $\Sigma'_\ell\to S$ has no ramification at the image 
of the  fixed point of the  lifted $\mmu_6$-action, whereas the common image of the remaining pair 
is a point of simple ramification. It follows that $D$ is  the discriminant of $\Sigma'_\ell\to S$ and that $\Sigma'_\ell$ has genus 4.  This also shows that $\Sigma'_\ell$ has no covering 
transformations. So $\Sigma'_\ell$ is unique up to unique isomorphism. 
\end{proof}

\begin{theorem}\label{thm:main}
The map which assigns to $[\ell]\in\PP (\FF_3\otimes L)$  the cover $\Sigma'_\ell\to S$
defines a $\Mod(S,D)$-equivariant bijection 
\[
[\ell]\in {\PP}(\FF_3\otimes L)\mapsto r(\Sigma'_\ell/S) \in R(S,D)
\]
In particular, $R(S,D)$ thus acquires the structure of the  projective space of a 10-dimensional symplectic space over $\FF_3$ on which $\G$ acts through  
a symplectic group modulo its center.
\end{theorem}
\begin{proof}
The $\Mod(S,D)$-equivariance of this assignment  is clear from the construction. The theorem then follows from the fact that the two sets have the same cardinality and that $\Mod(S,D)$ acts transitively on the target. The rest follows from Lemma \ref{lemma:mod3}. 
\end{proof}

\subsection*{The confluence divisors} Theorem \ref{thm:main} also explains why the preimage of $\Dcal\subset\BB_\G$ in $\BB_{\G'}$ has three irreducible components, to wit, the confluence divisors $\Dcal_\Hcal$, $\Dcal_{\ram}$ and $\Dcal_{\sg}$. Let us first observe that  a mirror is given by the $\Ecal$-span of  $(-3)$-vector of  $L$ and  that the mod $\theta$-reduction of  the latter
determines a line  in $\FF_3\otimes L$ and hence a point in $\PP(\FF_3\otimes L)$.

\begin{proposition}\label{prop:}
Let $\ell$ be a line in $\FF_3\otimes L$. Then the three orbits of $\G_\ell$ in 
$\PP(\FF_3\otimes L)$  are  $[\ell]$, $\PP(\ell^\perp)\ssm \{[\ell]\}$ and  $\PP(\FF_3\otimes L)\ssm \PP(\ell^\perp)$ and represent respectively  the irreducible components $\Dcal_\Hcal$, $\Dcal_{\ram}$ and $\Dcal_{\sg}$ of the preimage  of the confluence divisor $\Dcal$ of $\BB_\G$ in $\BB_{\G_\ell}$.
\end{proposition}
\begin{proof}
By Lemma  \ref{lemma:transitive} we can  choose a standard set  of arcs  $\{\g_i \}_{i\in \ZZ/12}$ connecting the base point $o$ with the points of $D$ such that that the 
associated collection $\{ c_i\in \pi_1(S\ssm D, o)\}_{i\in \ZZ/12}$ is a standard set of 
generators of $\pi_1(S\ssm D, o)$ (so this is just as in the proof of Lemma \ref{lemma:transitive}) for which the covering $\Sigma'\to S$  has the property that
$\rho(c_0)=\rho(c_1)\not=\rho(c_2)=\cdots =\rho (c_{11})$ (so this is different than in the proof of Lemma \ref{lemma:transitive}).  The juxtaposition of $\g_{i+1}$ and 
the inverse of  $\g_{i}$ is isotopic to an arc $\delta_i$ as in  Section \ref{sect:DMball}, and as explained there,  determines a $(-3)$-vector $a_{\delta_i}\in L$ up to 
a $\mmu_6$-multiple  and hence a line $\ell_i$ in $\FF_3\otimes L$.

Let $\ell$ be the line in $\FF_3\otimes L$ for which $\Sigma'\cong\Sigma'_\ell$. We
claim that $\ell=\ell_0$. Indeed, the covering  $\Sigma'_\ell\to S$ is disconnected over $S\ssm \delta_0$. It is then not hard to see  that $\Sigma_\ell \to \Sigma$ must be a trivial  $\ell^\vee$-covering over the preimage of $\delta_0$. This implies that $\ell$ has its support over $\delta_0$, so that $\ell=\ell_0$. 

It remains to observe that for $i=0,1,2$, the mirror $\BB^{T_{\gamma_i}}$ parametrizes the points where the covering $\Sigma'_\ell\to P$,  acquires, after shrinking $\delta_i$,  respectively a separating double point, a point of ramification of order $2$ and a nonseparating double point. So  $\ell_0$, $\ell_1$, $\ell_2$ represent respectively $\Dcal_\Hcal$, $\Dcal_{\ram}$ and $\Dcal_{\sg}$. 
\end{proof}

\subsection*{The strict transform of the thetanull locus} The goal of this section is to show that the strict transform 
$\Mcal_4(g^1_3)^0$ of $\Mcal_4^0$  in $\Mcal_4(g^1_3)$ is also locally a ball quotient in $\BB_{\G'}$.
This was actually established in the paper with Heckman \cite{HL} and the goal of this section is to show how. 
The generic point of that locus parametrizes the canonical genus $4$ curves $X$ that lie on a quadric cone  
$Q^0$ in a 3-dimensional complex projective space and are transversally cut out on $Q^0$ by a cubic hypersurface (so not containing the vertex). The base of the cone  $Q^0$ (in other words, the set of \emph{rays} in $Q^0$, lines on $Q^0$ that pass through its vertex) is a conic. We denote it by $P$, since the projection $X\to P$ gives us the unique $g^1_3$ on $X$. We assume for now that the discriminant divisor $D$ of this projection is reduced.

The vertex  of $Q^0$ is the unique singular point of $Q^0$ and is resolved by a single blowup: 
$\hat Q^0\to Q^0$.   Its exceptional set, which we can identify with $P$,  has self-intersection $-2$. 
It defines a section of the  evident  morphism $\hat Q^0\to P$ whose fibers are the  rays of $Q^0$. 
In other words, $\hat Q^0$ is a Hirzebruch  surface whose exceptional section is $P$. Its Picard group is freely generated by $P$ and the class $f$  of  fiber. The curve 
$X$ on $Q^0$ has class $3(P+2f)$ and hence the class of $X+P$ is $3(P+2f)+P=4P+6f$. Since this is $2$-divisible, it determines  admits a double covering $E_X\to \hat Q^0$ with ramification divisor 
$X+P$ (\footnote{Kond\=o \cite{kondo}  considers instead the $\mmu_3$-cover of $\hat Q^0$ ramified along $X$, which produces a K3-surface with $\mmu_3$-action.}). The composite with the projection $\hat Q^0\to P$ defines an elliptic fibration for which
the preimage of $P\subset \hat Q^0$ in $E_X$ defines the zero section 
(self-intersection number $-1$) for whose natural involution (inversion) is the one defined by 
the double cover $E_X\to \hat Q^0$. 
Note that $D$ now also appears as the discriminant of  the  fibration $E_X\to  P$. Indeed, all fibers 
are smooth, except those over $D$: over each point of $D$  we find a nodal curve (a Kodaira fiber of 
type $I_1$) whose double point lies over the singular point of the projection $X\to P$. 

By taking  the $j$-invariant of the fiber, we obtain a  morphism 
\[
j: P\to \PSL_2(\ZZ)\bs \HH^*=\big(\PSL_2(\ZZ)\bs \HH\big)\cup\{\infty\}=\overline\Mcal_{1,1},
\]
where $\HH^*:=\HH\cup \PP^1(\QQ)$ endowed with the horocyclic topology. This morphism has 
degree $12$ and $D=j^*(\infty)$. In  \cite{HL} it was observed that the $D$ that thus appear 
can be  characterized  in $\BB_\G$ as hyperball quotient as follows. The abelianization of 
$\PSL_2(\ZZ)$ is naturally isomorphic with $\mmu_6$ and hence defines  a ramified 
$\mmu_6$-cover 
$\PSL_2(\ZZ)'\bs\HH\to\PSL_2(\ZZ)\bs \HH$.  This morphism extends across the cusp 
$\infty$ with total ramification, so that we get a $\mmu_6$-cover
\[
\overline\Mcal{}'_{1,1}:=\PSL_2(\ZZ)'\bs\HH^*\to \PSL_2(\ZZ)\bs \HH^*=\overline\Mcal_{1,1}.
\]
The left hand side can be identified with the (Eisenstein) elliptic curve of genus zero $\CC/\Ecal$ whose origin is the unique point over $\infty$ and on which $\mmu_6$ acts in the standard manner. The differential  on $\CC/\Ecal$ defined by $dz$ transforms under the  $\mmu_6$-action via the character $\chi^{-1}$ and  hence $d\overline z$  transforms with  the character $\chi$. So $H^1(\overline\Mcal{}'_{1,1}; \CC)_\chi$ is of Hodge type $(0,1)$.

The $\mmu_6$-cover $C\to P$ with discriminant  $D$ is then the pull-back of the above $\mmu_6$-cover under the 
$j$-morphism:
\[
\begin{CD}
C@>{j'}>>\overline\Mcal{}'_{1,1}\\
@V{\mmu_6}VV@V{\mmu_6}VV\\
P@>j>> \overline\Mcal_{1,1}
\end{CD}
\]
Since $j$ is of degree $12$, so is $j'$.

\begin{lemma}\label{lemma:}
The map $j'{}^*$ embeds $H^1(\overline\Mcal{}'_{1,1})$ in $H^1(C)^\circ$ as a primitive rank one 
$\Ecal$-submodule  that is generated by a vector $\eps$ with  $h(\eps, \eps)=-6$. 
\end{lemma}
\begin{proof}
This is actually proved  in \cite{HL},  but let us  give here a simpler argument.
First note that since $H^1(\overline\Mcal{}'_{1,1})=H^1(\overline\Mcal{}'_{1,1})^\circ$, 
 $j'{}^*$ takes values in $H^1(C)^\circ$. 
If $a\in H^1(\overline\Mcal{}'_{1,1})$ is primitive, then $a \cdot \tau a=1$ and hence if we put
$\eps:=j'{}^*a$, then 
\[
h(\eps,\eps)=-\tfrac{1}{2} (\eps, \tau \eps)=-\tfrac{1}{2}.12 (a\cdot \tau a)=-6,
\] 
where we used that $j'$ is of degree $12$. This also implies that $\eps$ is primitive: if it is
divisible by $\lambda\in \Ecal$:  $\eps=\lambda\eps'$ with $\eps'\in H^1(C)^\circ$, then 
$\lambda\overline\lambda h(\eps' ,\eps')=-6$. But 
$\lambda\overline\lambda$ is a positive  integer  that cannot take the value $2$ and 
$h(\eps', \eps')\in \Ecal\theta\cap \QQ=3\ZZ$. 
This implies that $\lambda\overline\lambda=1$, meaning that $\lambda$ is a unit.
\end{proof}

Assertions A4 and A6 of \cite{HL} state that any $(-6)$-vector in $L$ can be written as the sum of two $(-3)$-vectors with inner product $\theta$ and that the $(-6)$-vectors make up a single $\G$-orbit. In particular, the hyperballs in $\BB$ defined  by such vectors define an irreducible totally geodesic divisor $\Dcal^0$ in $\BB_\G$. We can therefore complete Theorem \ref{thm:ballquotient2} as follows (see also \cite{HL}, Thm.\  8.2):

\begin{theorem}\label{thm:ballquotient3}
The isomorphism in Theorem \ref{thm:ballquotient2} takes the thetanull locus  $\Mcal_4(g^1_3)^0$ onto $\Dcal^0\ssm \Dcal_\sg$. 
In particular, $\Mcal_4(g^1_3)^0$ is a 
totally geodesic hypersurface in  $\Mcal_4(g^1_3)$. 
\end{theorem}

\begin{remark}\label{rem:}
One would expect a stronger assertion, namely that the  involution of $\Mcal_4(g^1_3)$ that assigns to $(X,P)$ the residual pair $(X,P')$ lifts to a reflection in $\G$ in a $(-6)$-vector $\eps$, in other words, is given by $x\in L\mapsto x +\frac{1}{3}h(x,\eps)\eps$. But such a reflection will not preserve $L$, as we must then have $h(x,\eps)\in 3\Ecal$. Indeed, since the $(-6)$-vectors lie in a single $\aut(L)$-orbit, it suffices to check this for one such vector, say 
$\eps:=a_{\delta_0}+a_{\delta_1}$. But $h(\eps,a_{\delta_2})=\theta$, which in $\Ecal$  is not divisible by $3$. It is surprising (and still a bit mystifying to us) that  $\BB_{\G'}\ssm \Dcal_\sg$ comes with an involution whose fixed point set is $\Dcal^0\ssm \Dcal_\sg$, but that is \emph{not} obtained from a subgroup $\hat\G'\subset \G$ with $[\hat\G':\G']=2$.
\end{remark}

\end{document}